\documentclass[a4paper,12pt,reqno]{amsart}
\usepackage[T1]{fontenc}
\usepackage[utf8]{inputenc}
\usepackage[libertinus,vvarbb]{newtx}
\usepackage[margin=1in]{geometry}
\usepackage{enumitem}
\usepackage{graphicx}
\usepackage[svgnames]{xcolor}
\usepackage{xparse}
\usepackage{xurl}
\usepackage[bookmarksdepth=2]{hyperref}

\hypersetup{colorlinks=true,urlcolor=MidnightBlue,linkcolor=MidnightBlue,citecolor=MidnightBlue}

\usepackage[normalem]{ulem}

\newcommand{\stkout}[1]{\ifmmode\text{\sout{\ensuremath{#1}}}\else\sout{#1}\fi}
\newcommand{\crsout}[1]{\ifmmode\text{\xout{\ensuremath{#1}}}\else\xout{#1}\fi}

\numberwithin{equation}{section}

\newtheorem{theorem}{Theorem}[section]
\newtheorem{lemma}[theorem]{Lemma}
\newtheorem{corollary}[theorem]{Corollary}

\theoremstyle{definition}

\newtheorem{remark}[theorem]{Remark}

\DeclareMathOperator{\Cl}{Cl}

\newcommand{\C}{\mathbb{C}}
\newcommand{\R}{\mathbb{R}}

\DeclareMathOperator{\dom}{dom}
\newcommand{\form}{\mathscr{E}}
\newcommand{\dir}{{\mathrm{D}}}
\newcommand{\neu}{{\mathrm{N}}}
\newcommand{\sub}{{\mathrm{S}}}
\newcommand{\fourier}{\mathscr{F}}
\newcommand{\ph}{\varphi}

\newcommand{\thet}{\vartheta}

\renewcommand{\le}{\leqslant}
\renewcommand{\ge}{\geqslant}

\NewDocumentCommand{\formula}{ssom}{%
 \IfBooleanTF{#1}{%
  \IfBooleanTF{#2}{%
   \IfValueTF{#3}%
    {\begin{align}\label{#3}\begin{gathered}#4\end{gathered}\end{align}}%
    {\begin{gather}#4\end{gather}}%
  }{%
   \IfValueTF{#3}%
    {\begin{align}\label{#3}\begin{aligned}#4\end{aligned}\end{align}}%
    {\begin{gather*}#4\end{gather*}}%
  }%
 }{%
  \IfValueTF{#3}%
   {\begin{align}\label{#3}#4\end{align}}%
   {\begin{align*}#4\end{align*}}%
 }%
}

\newcommand{\nist}[2]{\href{https://dlmf.nist.gov/#1\#E#2}{Eq.~#1.#2}}
\newcommand{\doi}[1]{\href{https://doi.org/#1}{\textsf{\scriptsize DOI:#1}}}
\newcommand{\arxiv}[1]{\href{https://arxiv.org/abs/#1}{\textsf{\scriptsize arXiv:#1}}}
\newcommand{\isbn}[1]{\textsf{ISBN:#1}}

\begin{document}

\title{Dirichlet--Neumann bracketing for nonlocal operators}
\author{Mateusz Kwaśnicki, Jacek Wszoła}
\thanks{Work supported by the National Science Centre, Poland, grant no.\@ 2023/49/B/ST1/04303}
\address{Mateusz Kwaśnicki, Jacek Wszoła \\ Department of Analysis and Stochastic Processes \\ Wrocław University of Science and Technology \\ ul. Wybrzeże Wyspiańskiego 27 \\ 50-370 Wrocław, Poland}
\email{\textsf{\href{mailto:mateusz.kwasnicki@pwr.edu.pl}{mateusz.kwasnicki@pwr.edu.pl}, \href{mailto:jacek.wszola@pwr.edu.pl}{jacek.wszola@pwr.edu.pl}}}

\begin{abstract}
We establish Dirichlet--Neumann bracketing for the Dirichlet eigenvalues of $\psi(-\Delta)$ on bounded Lipschitz domains, where $\psi$ is an arbitrary complete Bernstein function. The eigenvalues lie between $\psi$ applied to the corresponding Neumann and Dirichlet eigenvalues of the Laplacian. Both inequalities are strict whenever $\psi$ admits no meromorphic continuation to $\C \setminus \{0\}$. The proof uses quadratic forms, operator monotonicity, and an analysis of equality in resolvent comparisons. Applying the bracketing to intervals and balls gives a unified proof of simplicity of interval eigenvalues and antisymmetry of second eigenfunctions in balls under the same condition on $\psi$. For fractional powers, these recover results of Fall, Ghimenti, Micheletti and Pistoia for the interval, and of Fall, Feulefack, Temgoua and Weth and, independently, Benedikt, Bobkov, Dhara and Girg for the ball. The argument extends these conclusions to a broader class of nonlocal operators.
\end{abstract}

\maketitle

%
%

\section{Introduction}

Denote by $\Delta$ the Laplace operator in $\R^d$, and let $-(-\Delta)^{\alpha / 2}$ be the fractional Laplace operator of order $\alpha \in (0, 2)$. For a nonempty and bounded Lipschitz open set $D$, denote by $\Delta^\dir$ and $\Delta^\neu$ the Dirichlet and the Neumann Laplacians in $D$, respectively. Let $\mu_n^\dir$ and $\mu_n^\neu$ be the eigenvalues of $-\Delta^\dir$ and $-\Delta^\neu$, repeated according to their multiplicity and arranged in a nondecreasing order. Finally, denote by $(-(-\Delta)^{\alpha / 2})^\dir$ the Dirichlet fractional Laplace operator in $D$, and by $\lambda_n$ the nondecreasing sequence of eigenvalues of $((-\Delta)^{\alpha / 2})^\dir$, again repeated according to their multiplicity.

By a classical result of DeBlassie (Theorem~1.3 in~\cite{deblassie}) and Chen and Song (Theorem~3.3 in~\cite{cs:2}), we have
\formula{
 ((-\Delta)^{\alpha / 2})^\dir & \le (-\Delta^\dir)^{\alpha / 2}
}
in the sense of quadratic forms, and hence
\formula[eq:fractional:upper]{
 \lambda_n & \le (\mu_n^\dir)^{\alpha / 2}
}
for $n = 1, 2, \ldots$\, It was observed by Nazarov (Theorem~3 in~\cite{nazarov}) that we also have
\formula{
 ((-\Delta)^{\alpha / 2})^\dir & \ge (-\Delta^\neu)^{\alpha / 2} ,
}
which implies that
\formula[eq:fractional:lower]{
 \lambda_n & \ge (\mu_n^\neu)^{\alpha / 2}
}
for $n = 1, 2, \ldots$\, In this paper we observe that the Dirichlet--Neumann bracketing given by~\eqref{eq:fractional:upper} and~\eqref{eq:fractional:lower} leads to a very short and independent proof of two recent results resolving twenty-year-old conjectures of Bañuelos and Kulczycki regarding the simplicity of eigenvalues when $D$ is an interval (proved by Fall, Ghimenti, Micheletti and Pistoia in~\cite{fgmp}) and the type of symmetry of the second eigenfunction when $D$ is a ball (proved by Fall, Feulefack, Temgoua and Weth in~\cite{fftw}, and by Benedikt, Bobkov, Dhara and Girg in~\cite{bbdg}). The origins of these conjectures go back to the work of Bañuelos and Kulczycki~\cite{bk}, and we refer to~\cite{dfw,dkk:2,ferreira,kkms,kwasnicki:1,kwasnicki:2} for further partial results.

In fact, our approach allows us to replace the fractional power $\psi(\xi) = \xi^{\alpha / 2}$ by an arbitrary complete Bernstein function, leading to the following extensions of the results mentioned above.

\begin{theorem}[Dirichlet--Neumann bracketing]
\label{thm:main}
Let $\psi$ be a complete Bernstein function and $D$ a nonempty bounded Lipschitz domain in $\R^d$. Denote by $\mu_n^\neu$ and $\mu_n^\dir$ the Neumann and Dirichlet eigenvalues of $-\Delta$ in $D$, arranged in a nondecreasing order, and let $\lambda_n$ denote the Dirichlet eigenvalues of $\psi(-\Delta)$ in $D$. Then
\formula[eq:main]{
 \psi(\mu_n^\neu) & \le \lambda_n \le \psi(\mu_n^\dir)
}
for $n = 1, 2, \ldots\,$ The inequalities are strict unless $\psi$ extends to a meromorphic function in $\C \setminus \{0\}$.
\end{theorem}

The upper bound in~\eqref{eq:main} is well-known, see Theorem~3.3 in~\cite{cs:2}. Nevertheless, we provide a slightly different argument. The lower bound for $\psi(\xi) = \xi^{\alpha / 2}$ is given in Theorem~3 in~\cite{nazarov}, with a different proof.

\begin{theorem}[simplicity of eigenvalues in an interval]
\label{thm:interval}
Let $\psi$ be a complete Bernstein function and $D = (a, b)$ a finite interval in $\R$. Let $\lambda_n$ denote the Dirichlet eigenvalues of $\psi(-\Delta)$ in $D$, arranged in a nondecreasing order. Then
\formula[eq:interval]{
 \psi\bigl((\tfrac{(n - 1) \pi}{b - a})^2\bigr) & \le \lambda_n \le \psi\bigl((\tfrac{n \pi}{b - a})^2\bigr)
}
for $n = 1, 2, \ldots\,$ The inequalities are strict, and so $\lambda_n$ are simple, unless $\psi$ extends to a meromorphic function in $\C \setminus \{0\}$.
\end{theorem}

Simplicity of the eigenvalues of the fractional Laplace operator $-(-\Delta)^{\alpha / 2}$ in an interval was proved, with a different argument, in Corollary~2 in~\cite{fgmp}; see~\cite{kwasnicki} for a very short version of that proof.

\begin{theorem}
\label{thm:ball}
Let $D = B(0, 1)$ be the unit ball in $\R^d$, $d \ge 2$, and $\psi$ be a complete Bernstein function. Fix a nonzero solid harmonic polynomial $P$ of degree $\ell$. Denote by $\mu_{\ell, n}^\neu$ and $\mu_{\ell, n}^\dir$ the Neumann and Dirichlet eigenvalues of $-\Delta$ in $D$ that correspond to the symmetry class of $P$ (that is, with eigenfunctions of the form $P(x) \ph(\lvert x \rvert)$), arranged in a nondecreasing order. Similarly, let $\lambda_{\ell,n}$ denote the Dirichlet eigenvalues of $\psi(-\Delta)$ in $D$ that correspond to the symmetry class of $P$, arranged in a nondecreasing order. Then
\formula[eq:ball]{
 \psi(\mu_{\ell, n}^\neu) & \le \lambda_{\ell, n} \le \psi(\mu_{\ell, n}^\dir)
}
for $n = 1, 2, \ldots\,$ The inequalities are strict unless $\psi$ extends to a meromorphic function in $\C \setminus \{0\}$.
\end{theorem}

\begin{corollary}[antisymmetry of the second eigenfunction in a ball]
\label{cor:shape}
Under the assumptions of Theorem~\ref{thm:ball}, we have
\formula[eq:shape]{
 \lambda_{1, 1} \le \psi(\mu_{1, 1}^\dir) = \psi(\mu_{0, 2}^\neu) \le \lambda_{0, 2} .
}
Hence, the second smallest eigenvalue of $\psi(-\Delta)$ in the unit ball (counting multiplicities) corresponds to an antisymmetric eigenfunction. Unless $\psi$ extends to a meromorphic function in $\C \setminus \{0\}$, the inequalities in~\eqref{eq:shape} are strict, and every eigenfunction corresponding to the second smallest eigenvalue is antisymmetric.
\end{corollary}

Antisymmetry of the second eigenfunction of the fractional Laplace operator $-(-\Delta)^{\alpha / 2}$ in a ball was proved, using different methods, in Theorem~1.2 in~\cite{fftw}, and in Theorem~1.1 in~\cite{bbdg}.

To our knowledge, the Neumann lower bound and the sufficient condition for strict bracketing in Theorems~\ref{thm:main} and~\ref{thm:ball} are new for general complete Bernstein functions. They yield the extensions of interval simplicity and ball antisymmetry stated in Theorem~\ref{thm:interval} and Corollary~\ref{cor:shape}.

\begin{remark}
For the upper bound in~\eqref{eq:main}, the Laplace operator $-\Delta$ in the above result can be replaced by an arbitrary self-adjoint operator $L$ associated to a regular Dirichlet form $\form$, and $D$ can be an arbitrary nonempty open set; see~\cite{cs:2} for further discussion. In this case, the operators $L$ and $\psi(L)$ with Dirichlet boundary/exterior conditions correspond to the closures of restrictions of the quadratic forms of $L$ and $\psi(L)$ to the class of continuous functions with a compact support contained in $D$.

The lower bound in~\eqref{eq:main} is also quite general. It extends to a wider class of operators $L$, which are associated to regular Dirichlet forms $\form$. The quadratic form of the corresponding operator $L^\neu$ on $D$ with Neumann boundary conditions is then defined by restricting the domain of integration in the Beurling--Deny representation of $\form$ to $D$. Our methods apply if the quadratic form $\form^\neu$ of $L^\neu$ defined above is closed, $\form^\neu \le \form$, $L^2(D)$ is an invariant space for $L^\neu$, and the resolvent operators of $L^\neu$ are compact on $L^2(D)$.

Strict inequalities in~\eqref{eq:main} depend on the following property of $\Delta$: there is no nontrivial closed subspace of $L^2(\R^d)$ which is invariant for the resolvent of $\Delta$ and which consists of functions vanishing in the complement of $D$. Our proof applies to more general operators $L$ with this property.
\end{remark}

\begin{remark}
In~\cite{cs:1}, Caffarelli and Silvestre proved that the fractional Laplace operator $(-\Delta)^{\alpha / 2}$ is the Dirichlet-to-Neumann map for an appropriate second-order differential operator $\mathscr L$ in the half-space $\R^d \times [0, \infty)$. In the same vein, the fractional Dirichlet and Neumann Laplace operators $(-\Delta^\dir)^{\alpha / 2}$ and $(-\Delta^\neu)^{\alpha / 2}$ are the Dirichlet-to-Neumann maps for the same differential operator $\mathscr L$ in $D \times [0, \infty)$, with a Dirichlet or Neumann boundary condition imposed on $(\partial D) \times (0, \infty)$. The Dirichlet--Neumann bracketing~\eqref{eq:main} follows in this case from a simple comparison of the quadratic forms of $\mathscr L$ in these three settings: on $\R^d \times [0, \infty)$, on $D \times [0, \infty)$ with the Dirichlet boundary condition, and on $D \times [0, \infty)$ with the Neumann boundary condition. Furthermore, strict inequalities can be obtained using the unique continuation property. This approach was taken in~\cite{mn:1,mn:2,nazarov}.

When $\psi$ is an arbitrary complete Bernstein function, a similar description of the operator $\psi(-\Delta)$ as the Dirichlet-to-Neumann map for a (possibly singular, generalised) second-order differential operator $\mathscr L$ in the half-space $\R^d \times [0, \infty)$ was found in~\cite{km}. It is quite clear that $\psi(-\Delta^\dir)$ and $\psi(-\Delta^\neu)$ admit analogous representations (see~\cite{atw,ah,gms,lee,st,tan}). Thus, an argument sketched in the previous paragraph likely leads to an alternative proof of weak inequalities in Theorem~\ref{thm:main}. However, as we expect various technical difficulties related to highly singular differential operators $\mathscr L$ involved, we decided to take a different approach.
\end{remark}

\begin{remark}
Probabilistically, $\Delta$ is the generator of the Brownian motion in $\R^d$ (with covariance matrix $2 t I$), while $\Delta^\dir$ and $\Delta^\neu$ are the generators of the killed and reflected Brownian motions in $D$, respectively. The operators $-\psi(-\Delta)$, $-\psi(-\Delta^\dir)$ and $-\psi(-\Delta^\neu)$ are the generators of subordinate processes, where $\psi$ is the Laplace exponent of the corresponding (complete) subordinator (killed at the uniform rate $\psi(0)$ if $\psi(0) > 0$). Finally, $-(\psi(-\Delta))^\dir$ is the generator of the killed subordinate process.
\end{remark}

%
%

\section{Preliminaries}

Throughout this work, $d = 1, 2, \ldots$ is a dimension. By $\fourier$ we denote the Fourier transform on $L^2(\R^d)$, that is, the continuous extension to $L^2(\R^d)$ of the operator initially defined for $u \in L^2(\R^d) \cap L^1(\R^d)$ by
\formula{
 \fourier u(\xi) & = \int_{\R^d} e^{-i \xi \cdot x} u(x) dx .
}
Below we give rigorous definitions of operators discussed in the introduction and collect auxiliary results.


\subsection{Quadratic forms}

We consider a possibly unbounded nonnegative definite self-adjoint operator $L$ on a closed subspace $H_L$ of a Hilbert space $H$. Let $\form$ be the corresponding nonnegative, closed, lower semi-continuous quadratic form on $H$. Denote by $\dom(L)$ and $\dom(\form)$ the domains of $L$ and $\form$, respectively. For $u \in \dom(L)$, we have
\formula{
 \form(u) & = \langle L u, u \rangle .
}
Since $L$ is a possibly unbounded nonnegative definite self-adjoint operator on the Hilbert space $H_L = \Cl \dom(L) = \Cl \dom(\form)$, we can define the square root operator $L^{1 / 2}$ on $H_L$ by means of the spectral theorem. Then $\dom(\form)$ coincides with $\dom(L^{1 / 2})$, and for $u \in \dom(\form)$,
\formula{
 \form(u) & = \lVert L^{1 / 2} u \rVert^2 .
}
We define $\form(u) = \infty$ when $u \in H \setminus \dom(\form)$. The correspondence between $L$ and $\form$ is one-to-one. These and many more standard properties of self-adjoint operators and their associated quadratic forms can be found, for example, in~\cite{ouhabaz}.

If $s > 0$, then the resolvent operator $(L + s)^{-1}$ is a bounded self-adjoint operator on $H_L$, with operator norm at most $s^{-1}$; see Proposition~1.22 in~\cite{ouhabaz}. We extend $(L + s)^{-1}$ to a bounded self-adjoint operator on $H$ so that $(L + s)^{-1} u = 0$ for $u$ in the orthogonal complement of $H_L$.

Let $L_1$ and $L_2$ be possibly unbounded nonnegative definite self-adjoint operators on closed subspaces $H_1$ and $H_2$ of a Hilbert space $H$. Denote the corresponding quadratic forms on $H$ by $\form_1$ and $\form_2$. We write $L_1 \le L_2$ if $\form_1 \le \form_2$ on $H$. The latter property implies that $\dom(\form_2) \subseteq \dom(\form_1)$.

\begin{lemma}[Proposition~3.1 in~\cite{vv}; Lemma~3.2 in~\cite{hssw}]
\label{lem:rm}
If $L_1$ and $L_2$ are possibly unbounded nonnegative definite self-adjoint operators on closed subspaces of a given Hilbert space, then
\formula{
 & L_1 \le L_2 \text{ implies } (L_1 + s)^{-1} \ge (L_2 + s)^{-1} \text{ for every } s > 0 .
}
\end{lemma}

\begin{proof}[Sketch of the proof]
When $H_1 = H_2$, this is a standard result, usually proved by showing that $(L_1 + s)^{1 / 2} (L_2 + s)^{-1 / 2}$ is a contraction on $H_1 = H_2$:
\formula{
 \lVert (L_1 + s)^{1 / 2} (L_2 + s)^{-1 / 2} u \rVert^2 & = \form_1((L_2 + s)^{-1 / 2} u) + s \lVert (L_2 + s)^{-1 / 2} u \rVert^2 \\
 & \le \form_2((L_2 + s)^{-1 / 2} u) + s \lVert (L_2 + s)^{-1 / 2} u \rVert^2 \\
 & = \lVert (L_2 + s)^{1 / 2} (L_2 + s)^{-1 / 2} u \rVert^2 = \lVert u \rVert^2 .
}
It follows that also the adjoint operator $((L_1 + s)^{1 / 2} (L_2 + s)^{-1 / 2})^*$ is a contraction. Hence, with $v = (L_1 + s)^{-1 / 2} u$,
\formula{
 \langle (L_2 + s)^{-1} u, u \rangle & = \lVert (L_2 + s)^{-1/2} u \rVert^2 \\
 & = \lVert (L_2 + s)^{-1/2} (L_1 + s)^{1 / 2} v \rVert^2 \\
 & = \lVert ((L_1 + s)^{1 / 2} (L_2 + s)^{-1 / 2})^* v \rVert^2 \\
 & \le \lVert v \rVert^2 = \lVert (L_1 + s)^{-1 / 2} u \rVert^2 = \langle (L_1 + s)^{-1} u, u \rangle .
}
The more general statement, where $H_1$ need not be equal to $H_2$, is technically more demanding. It is a special case of Proposition~3.1 in~\cite{vv} or Lemma~3.2 in~\cite{hssw}, and we omit the proof.
\end{proof}

\begin{remark}
\label{rem:rm:strict}
Under the assumptions of Lemma~\ref{lem:rm}, if $H_1 = H_2$ and we additionally have a strict inequality $\form_1(u) < \form_2(u)$ for all nonzero $u \in \dom(\form_1) \cap \dom(\form_2)$, then the above proof shows that $\langle (L_2 + s)^{-1} u, u \rangle < \langle (L_1 + s)^{-1} u, u \rangle$ for every nonzero $u \in H_1 = H_2$.

Written more compactly, but with a slight abuse of notation: if $H_1 = H_2$ and $L_1 < L_2$ on $H_1 = H_2$, then $(L_1 + s)^{-1} > (L_2 + s)^{-1}$ on $H_1 = H_2$.
\end{remark}

Recall that a subspace $H_0$ of $L^2(\R^d)$ is said to be invariant for $L$, or a reducing subspace for $L$, if the resolvent operators $(s + L)^{-1}$ map $H_0$ to $H_0$ for every $s \in \C \setminus (-\infty, 0]$. In particular, $H_L$ and its orthogonal complement are invariant for $L$.


\subsection{The Courant--Fischer theorem}

As before, let $L$ be a possibly unbounded nonnegative definite self-adjoint operator on a closed subspace $H_L$ of a separable Hilbert space $H$, and let $\form$ denote the corresponding quadratic form. Suppose that the spectrum of $L$ is purely discrete. Then the $n$th smallest eigenvalue $\mu_n$ (counting multiplicities) is given by the Courant--Fischer min-max formula:
\formula*[eq:cf]{
 \mu_n & = \min \bigl\{ \max \{\form(u) \, : \, u \in E , \, \lVert u \rVert_2^2 = 1\} \; : \\
 & \qquad\qquad \text{$E$ is an $n$-dimensional subspace of $H$} \bigr\} .
}
Of course, we can replace $H$ by $\dom(\form)$ in~\eqref{eq:cf}. If the dimension of $\dom(\form)$ is less than $n$, we leave $\mu_n$ undefined.

Let $\ph_n$ be the corresponding orthonormal eigenvectors: $L \ph_n = \mu_n \ph_n$. The minimum in~\eqref{eq:cf} is attained when $E$ is the linear span of $\ph_1, \ph_2, \ldots, \ph_n$. Conversely, if the minimum in~\eqref{eq:cf} is attained for $E$, then $E$ contains an eigenvector with eigenvalue $\mu_n$. We refer to~\cite{ch,rs:4} for further discussion, variants, and applications.

Suppose that $L_1$, $L_2$ are as above, and denote by $\form_1$, $\form_2$ the corresponding quadratic forms, by $\mu_{\smash{n}}^{(1)}$, $\mu_{\smash{n}}^{(2)}$ the corresponding eigenvalues arranged in a nondecreasing order. By~\eqref{eq:cf}, clearly
\formula{
 L_1 & \le L_2 \text{ implies } \mu_n^{(1)} \le \mu_n^{(2)} \text{ for } n = 1, 2, \ldots
}
Furthermore, if $\mu_{\smash{n}}^{(1)} = \mu_{\smash{n}}^{(2)}$ for some $n$, then there is a common eigenvector of $L_1$ and $L_2$ with eigenvalue $\mu_{\smash{n}}^{(1)} = \mu_{\smash{n}}^{(2)}$.


\subsection{Complete Bernstein functions and subordinate operators}

A function $\psi : [0, \infty) \to [0, \infty)$ is said to be a complete Bernstein function if
\formula[eq:cbf]{
 \psi(\xi) & = c_1 \xi + c_0 + \frac{1}{\pi} \int_{(0, \infty)} \frac{\xi}{\xi + s} \, \mu(ds) ,
}
where $c_1, c_0 \ge 0$ and $\mu$ is a nonnegative Borel measure such that the above integral is finite for some (or, equivalently, for all) $\xi > 0$; see Chapters~7 and~8 in~\cite{ssv}. Note that~\eqref{eq:cbf} defines a holomorphic function of $\xi \in \C \setminus (-\infty, 0]$. We will need the following simple observation: $\psi$ extends further to a meromorphic function on $\C \setminus \{0\}$ if and only if $\mu$ is a purely atomic measure on $(0, \infty)$, with atoms forming a discrete subset of $(0, \infty)$. The latter condition means that the support of $\mu$ possibly contains sequences convergent to $0$ or at $\infty$, but it has no accumulation points in $(0, \infty)$.

If $L$ is a nonnegative self-adjoint possibly unbounded operator on a closed subspace $H_L$ of a Hilbert space $H$, then the subordinate operator $L^\sub = \psi(L)$ on $H_L$ can be defined by means of the spectral theorem. If $\psi$ has a finite limit $\psi(\infty)$ at infinity, we extend the definition of $L^\sub$ to $H$ by setting $L^\sub u = \psi(\infty) u$ for $u$ in the orthogonal complement of $H_L$; otherwise, we leave $L^\sub$ only defined on a subspace of $H_L$.

An equivalent definition of $L^\sub$ involves the resolvent operators $(L + s)^{-1}$: we have
\formula[eq:balakrishnan]{
 L^\sub u & = c_1 L u + c_0 u + \frac{1}{\pi} \int_{(0, \infty)} \bigl( u - s (L + s)^{-1} u \bigr) \mu(ds) .
}
The domain of $L^\sub$ is the same as the domain of $L$ if $c_1 > 0$. Otherwise, $\dom(L^\sub)$ contains $\dom(L)$, and~\eqref{eq:balakrishnan} holds on $\dom(L)$, with an absolutely convergent Bochner's integral on the right-hand side. The above formula is an extension of Balakrishnan's representation of fractional powers of operators; see Chapter~3 in~\cite{ms} or Section~13.2 in~\cite{ssv}. Since $\xi / (\xi + s) = 1 - s (\xi + s)^{-1}$, formula~\eqref{eq:balakrishnan} is a perfect analogue of~\eqref{eq:cbf}. 

We will not need~\eqref{eq:balakrishnan} in this paper, but we will use an equivalent expression for the quadratic form $\form^\sub$ associated to the subordinate operator $L^\sub$:
\formula[eq:form:balakrishnan]{
 \form^\sub(u) & = c_1 \form(u) + c_0 \lVert u \rVert^2 + \frac{1}{\pi} \int_{(0, \infty)} \bigl( \lVert u \rVert^2 - \langle s (L + s)^{-1} u, u \rangle \bigr) \mu(ds) .
}
Here we agree $0 \cdot \infty = 0$ when $c_1 = 0$ and $\form(u) = \infty$. The corresponding domain $\dom(\form^\sub)$ is equal to $\dom(\form)$ if $c_1 > 0$, and otherwise it consists of those $u \in H$ for which the above integral is finite.

Equivalence of~\eqref{eq:form:balakrishnan} and the spectral-theoretic definition of the quadratic form of $L^\sub$ is likely well-known, but difficult to find in the literature, so let us sketch the proof. On $H_L$, formula~\eqref{eq:form:balakrishnan} follows directly from the spectral theorem. If $\psi(\infty) = \infty$, then both sides of~\eqref{eq:form:balakrishnan} are infinite on the complement of $H_L$. Otherwise, if $\psi(\infty)$ is finite, the validity of~\eqref{eq:form:balakrishnan} on the orthogonal complement of $H_L$ is a direct consequence of the definition. Finally, if $u = u_1 + u_2$, where $u_1$ is in $H_L$ and $u_2$ is in the orthogonal complement of $H_L$, then $\form^\sub(u) = \form^\sub(u_1) + \form^\sub(u_2)$ (because $H_L$ and its orthogonal complement are invariant subspaces for $L^\sub$), and the right-hand side of~\eqref{eq:form:balakrishnan} is easily seen to have a similar additivity property.

It is well-known that the class of complete Bernstein functions coincides with the class of operator monotone functions. That is, if $L_1$ and $L_2$ are nonnegative definite self-adjoint matrices and $\psi$ is a complete Bernstein function, then
\formula{
 & L_1 \le L_2 \text{ implies } \psi(L_1) \le \psi(L_2) .
}
Conversely, if a continuous function $\psi : [0, \infty) \to [0, \infty)$ has the above property for all nonnegative definite self-adjoint matrices $L_1, L_2$, then $\psi$ is a complete Bernstein function. Operator monotonicity extends to self-adjoint operators on Hilbert spaces; see Chapter~14 in~\cite{ssv} for a detailed discussion. The following simple lemma extends operator monotonicity to our more general setting of self-adjoint operators on different closed subspaces of a given Hilbert space $H$.

\begin{lemma}
\label{lem:om}
If $L_1, L_2$ are nonnegative definite possibly unbounded self-adjoint operators on closed subspaces of a Hilbert space $H$ and $\psi$ is a complete Bernstein function, then
\formula{
 & L_1 \le L_2 \text{ implies } L_1^\sub \le L_2^\sub .
}
\end{lemma}

\begin{proof}
Suppose that $L_1 \le L_2$, and denote by $\form_{\smash{1}}^\sub$ and $\form_{\smash{2}}^\sub$ the quadratic forms associated to $L_{\smash{1}}^\sub$ and $L_{\smash{2}}^\sub$. The desired inequality $\form_{\smash{1}}^\sub(u) \le \form_{\smash{2}}^\sub(u)$ follows directly from the representation~\eqref{eq:form:balakrishnan} for $\form_{\smash{1}}^\sub$ and $\form_{\smash{2}}^\sub$, combined with the assumption $\form_1 \le \form_2$ and Lemma~\ref{lem:rm}.
\end{proof}

We will need the following addition to the above lemma for strict inequalities.

\begin{lemma}
\label{lem:om:strict}
Let $\psi$ be a complete Bernstein function which does not extend to a meromorphic function on $\C \setminus \{0\}$. Suppose that $L_1, L_2$ are nonnegative definite possibly unbounded self-adjoint operators on closed subspaces $H_1, H_2$ of a Hilbert space $H$, and $L_1 \le L_2$. Denote by $\form_1$, $\form_2$, $\form_{\smash{1}}^\sub$ and $\form_{\smash{2}}^\sub$ the quadratic forms of $L_1$, $L_2$, $L_{\smash{1}}^\sub$ and $L_{\smash{2}}^\sub$, respectively. If $u \in \dom(\form_{\smash{2}}^\sub) \cap H_2$, $u \ne 0$ and $\form_{\smash{1}}^\sub(u) = \form_{\smash{2}}^\sub(u)$, then $u$ belongs to a closed subspace of $H$, the closed linear span of
\formula[eq:om:strict:cyclic]{
 \{ (L_1 + s)^{-1} u : s > 0 \} & = \{ (L_2 + s)^{-1} u : s > 0 \} ,
}
which is invariant for both $L_1$ and $L_2$, and on which $\form_1 = \form_2$.
\end{lemma}

\begin{proof}
Suppose that $u \in \dom(\form_{\smash{2}}^\sub)$ and $\form_{\smash{1}}^\sub(u) = \form_{\smash{2}}^\sub(u)$. By~\eqref{eq:form:balakrishnan}, for almost every $s > 0$ with respect to $\mu$, we have
\formula[eq:om:resolvent]{
 \lVert u \rVert^2 - \langle s (L_1 + s)^{-1} u, u \rangle & = \lVert u \rVert^2 - \langle s (L_2 + s)^{-1} u, u \rangle .
}
By our assumption, the support of $\mu$ has an accumulation point in $(0, \infty)$. Since both sides of~\eqref{eq:om:resolvent} are holomorphic functions of $s$ in $\C \setminus (-\infty, 0]$, equality in~\eqref{eq:om:resolvent} extends to all $s \in \C \setminus (-\infty, 0]$.

It follows that for every $s > 0$,
\formula{
 \bigl\langle \bigl( (L_1 + s)^{-1} - (L_2 + s)^{-1} \bigr) u, u \bigr\rangle & = 0 .
}
By Lemma~\ref{lem:rm}, $T = (L_1 + s)^{-1} - (L_2 + s)^{-1}$ is a nonnegative definite operator. Therefore, $\langle T u, u \rangle = 0$ implies that $T u = 0$. We conclude that
\formula{
 (L_1 + s)^{-1} u & = (L_2 + s)^{-1} u
}
for every $s > 0$. Let $H_0$ be the closed linear span of the vectors $(L_1 + s)^{-1} u = (L_2 + s)^{-1} u$, where $s > 0$ (the cyclic space of $u$). Since $u \in H_2$, $u$ is the limit of $s (L_2 + s)^{-1} u$ as $s \to \infty$, and so $u \in H_0$. Using the resolvent equation, we find that $H_0$ is invariant for both $L_1$ and $L_2$. Furthermore, by the uniqueness of the Cauchy transform, the projection-valued spectral measures of $L_1$ and $L_2$ applied to $u$ are equal, and so the spectral theorem implies that the quadratic forms of $L_1$ and $L_2$ coincide on $H_0$; we refer to Chapters~VII and~VIII in~\cite{rs:1} for further discussion.
\end{proof}


\subsection{The Laplace operator}

Let $\Delta$ be the Laplace operator on $\R^d$, and denote by $\form$ the quadratic form associated with $L = -\Delta$. The domain of $\form$ is the Sobolev space $\dom(\form) = H^1(\R^d)$, which consists of functions $u \in L^2(\R^d)$ such that the weak partial derivatives of $u$ are well-defined and belong to $L^2(\R^d)$. For $u \in \dom(\form)$, we have
\formula{
 \form(u) & = \int_{\R^d} \lvert \nabla u(x) \rvert^2 dx .
}
Recall that $\form(u) = \infty$ when $u \notin \dom(\form)$. In Fourier space,
\formula{
 \form(u) & = \frac{1}{(2 \pi)^d} \int_{\R^d} \lvert \xi \rvert^2 \lvert \fourier u(\xi) \rvert^2 d\xi ,
}
with $u \in \dom(\form)$ if and only if $u \in L^2(\R^d)$ and the above integral is finite. The class $C_c^\infty(\R^d)$ of smooth, compactly supported functions is a core of $\form$: every $u \in \dom(\form)$ can be approximated by a sequence $u_n \in C_c^\infty(\R^d)$ so that both $\lVert u_n - u \rVert_2$ and $\form(u - u_n)$ converge to zero as $n \to \infty$.

We claim that if a closed subspace $H_0$ of $L^2(\R^d)$ is invariant for $L$ and contains only functions which are equal to zero outside a given bounded set $D$, then $H_0$ is trivial. Indeed: if $u \in H_0$ is nonzero, then the Fourier transform of $u$ extends to a nonzero entire function on $\C^d$. There exists $\tilde \xi \in \C^d \setminus \R^d$ such that $\fourier u(\tilde \xi) \ne 0$ and $s = -(\tilde \xi_1^2 + \tilde \xi_2^2 + \ldots + \tilde \xi_d^2)$ belongs to $\C \setminus (-\infty, 0]$. If $v = (s + L)^{-1} u$, then $v \in H_0$ and $\fourier v(\xi) = (s + \xi_1^2 + \xi_2^2 + \ldots + \xi_d^2)^{-1} \fourier u(\xi)$. This means that $\fourier v$ is not an entire function (it has a pole at $\tilde \xi$), so $v$ cannot be equal to zero almost everywhere in $\R^d \setminus D$. This proves our claim.

Let $D$ be a nonempty open subset of $\R^d$. We denote by $\Delta^\dir$ the Dirichlet Laplace operator in $D$. By this we mean that the quadratic form of $L^\dir = -\Delta^\dir$ is the restriction of $\form$ to the class $\dom(\form^\dir)$ of functions in $\dom(\form)$ which are equal to zero almost everywhere in $\R^d \setminus D$. 

Note that we can view $\Delta^\dir$ as an unbounded operator on $L^2(\R^d)$, which is, however, not densely defined. Similarly, we may regard $\form^\dir$ as a quadratic form on $L^2(\R^d)$, the domain of which is again not dense in $L^2(\R^d)$. We will always take this perspective, in contrast to the more customary approach, where the Dirichlet Laplace operator acts on functions defined only on $D$. This allows us to write $L \le L^\dir$, and so, by Lemma~\ref{lem:rm}, $(L + s)^{-1} \ge (L^\dir + s)^{-1}$ for every $s > 0$.

\begin{remark}
\label{rem:dirichlet}
Traditionally, $\Delta^\dir$ is defined as the Friedrichs extension of $\Delta$ restricted to $C_c^\infty(D)$. In sufficiently regular domains, such as bounded Lipschitz open sets considered here, this is equivalent to the simplistic definition given above, which is additionally more convenient for our needs. We refer to Section~4.4 in~\cite{fot} for a detailed discussion.
\end{remark}

The Neumann Laplace operator $\Delta^\neu$ in a Lipschitz open set $D$ is also specified by prescribing the quadratic form $\form^\neu$ of $L^\neu = -\Delta^\neu$. We define the domain $\dom(\form^\neu)$ to be the class of $u \in L^2(\R^d)$ which are equal in $D$ to some $\tilde u \in \dom(\form) = H^1(\R^d)$. In this case, we set
\formula{
 \form^\neu(u) & = \int_D \lvert \nabla \tilde u(x) \rvert^2 dx .
}
For a detailed discussion, we refer to Example~1.6.1 in~\cite{fot}.

This definition of $\form^\neu$ is again slightly non-standard, as it is customary to consider functions $u$ defined only on $D$, while we allow $u$ to take arbitrary values in $\R^d \setminus D$. We define the Neumann Laplace operator $\Delta^\neu$ to be equal to $-L^\neu$, where $L^\neu$ is the operator associated to $\form^\neu$. Note that with this definition all functions $u \in L^2(\R^d)$ which are equal to zero in $D$ belong to $\dom(L^\neu)$, and we have $\Delta^\neu u = 0$ in $\R^d \setminus D$ whenever $u \in \dom(\Delta^\neu)$. Clearly, $L \ge L^\neu$, and so, by Lemma~\ref{lem:rm}, $(L + s)^{-1} \le (L^\neu + s)^{-1}$ for every $s > 0$.


\subsection{Subordinate Laplace operators}

We fix a complete Bernstein function $\psi$. Since Fourier transform turns $L = -\Delta$ into a multiplication operator, the subordinate operator $L^\sub = \psi(-\Delta)$ is the self-adjoint operator associated to the quadratic form
\formula{
 \form^\sub(u) & = \frac{1}{(2 \pi)^d} \int_{\R^d} \psi( \lvert \xi \rvert^2 ) \lvert \fourier u(\xi) \rvert^2 d\xi ,
}
and the domain $\dom(\form^\sub)$ consists of those $u \in L^2(\R^d)$ for which the integral on the right-hand side is finite. Equivalently,
\formula{
 \form^\sub(u) & = c_1 \form(u) + c_0 \lVert u \rVert_2^2 + \frac{1}{2} \int_{\R^d} \int_{\R^d} \lvert u(y) - u(x) \rvert^2 \nu(x - y) dx dy ,
}
where, with $p_t(z) = (4 \pi t)^{-d / 2} e^{-\lvert z \rvert^2 / (4 t)}$ denoting the Gauss--Weierstrass kernel, $\nu$ is given by
\formula{
 \nu(z) & = \frac{1}{\pi} \int_0^\infty \biggl( \int_{(0, \infty)} s e^{-s t} \mu(ds) \biggr) p_t(z) dt ;
}
see Chapter~15 in~\cite{ssv}.

We have the following simple result. It exploits our non-standard definitions of $L^\neu$ and $L^\dir$ as operators on $L^2(\R^d)$ rather than $L^2(D)$.

\begin{lemma}
\label{lem:ns:s:ds}
If $D$ is a Lipschitz open set in $\R^d$, then, with the notation introduced above,
\formula[eq:ns:s:ds]{
 0 & \le L^{\neu\sub} \le L^\sub \le L^{\dir\sub} .
}
\end{lemma}

\begin{proof}
Recall that for every $u \in L^2(\R^d)$, we have
\formula{
 0 & \le \form^\neu(u) \le \form(u) \le \form^\dir(u) ,
}
that is, $0 \le L^\neu \le L \le L^\dir$. Since $\psi$ is operator monotone, \eqref{eq:ns:s:ds} follows by Lemma~\ref{lem:om}.
\end{proof}

Let $D$ be a nonempty open subset of $\R^d$. We define the Dirichlet subordinate operator $L^{\sub\dir}$ in terms of its quadratic form: $\form^{\sub\dir}$ is the restriction of $\form^\sub$ to the class $\dom(\form^{\sub\dir})$ of functions in $\dom(\form^\sub)$ which are equal to zero almost everywhere in $\R^d \setminus D$. This is completely analogous to the definition of the Dirichlet Laplace operator, and Remark~\ref{rem:dirichlet} applies also to $L^{\sub\dir}$. As usual, we set $\form^{\sub\dir}(u) = \infty$ for $u \in L^2(\R^d) \setminus \dom(\form^{\sub\dir})$.

We have the following supplement to Lemma~\ref{lem:ns:s:ds}.

\begin{lemma}
\label{lem:s:sd:ds}
If $D$ is a Lipschitz open set in $\R^d$ and $\psi(\infty) = \infty$, then, with the notation introduced above,
\formula[eq:s:sd:ds]{
 L^\sub \le L^{\sub\dir} \le L^{\dir\sub} .
}
If $\psi(\infty) < \infty$, then~\eqref{eq:s:sd:ds} holds, in the sense of the inequality of the corresponding forms, on the class of functions in $L^2(\R^d)$ which are equal to zero almost everywhere in $\R^d \setminus D$.
\end{lemma}

\begin{proof}
The first inequality in~\eqref{eq:s:sd:ds} follows directly from the definition of $L^{\sub\dir}$: we have $\form^{\sub\dir}(u) = \form^\sub(u)$ if $u \in \dom(\form^{\sub\dir})$, and $\form^{\sub\dir}(u) = \infty \ge \form^\sub(u)$ if $u \in L^2(\R^d) \setminus \dom(\form^{\sub\dir})$.

In the same way, $\form \le \form^\dir$. Let, not unexpectedly, $\form^{\dir\sub\dir}(u) = \form^{\dir\sub}(u)$ for $u \in \dom(\form^{\dir\sub})$ which are equal to zero almost everywhere in $\R^d \setminus D$, and $\form^{\dir\sub\dir}(u) = \infty$ otherwise. From $\form \le \form^\dir$ and~\eqref{eq:form:balakrishnan} it follows that $\form^\sub \le \form^{\dir\sub}$, and so $\form^{\sub\dir} \le \form^{\dir\sub\dir}$. We have thus shown that
\formula{
 \form^\sub \le \form^{\sub\dir} \le \form^{\dir\sub\dir} .
}
It remains to note that, again by~\eqref{eq:form:balakrishnan}, $\form^{\dir\sub\dir}(u) = \form^{\dir\sub}(u)$ if either $\psi(\infty) = \infty$ and $u \in L^2(\R^d)$ is arbitrary, or $\psi(\infty) < \infty$ and $u \in L^2(\R^d)$ is equal to zero almost everywhere in $\R^d \setminus D$.
\end{proof}

%
%

\section{Eigenvalue estimates}


In this section we prove Theorem~\ref{thm:main}. We use the notation introduced in the previous section, and we consider operators acting on closed subspaces of $L^2(\R^d)$. More precisely, many operators below act on the closed subspace of $L^2(\R^d)$ consisting of functions equal to zero in $\R^d \setminus D$, where $D$ is a given open set. We abuse the notation and identify this subspace with $L^2(D)$; thus,
\formula{
 L^2(D) & = \{ u \in L^2(\R^d) \, : \, u = 0 \text{ in } \R^d \setminus D \} .
}

Suppose that $D$ is a bounded nonempty open subset of $\R^d$. Standard compactness arguments show that the operator $L^\dir = -\Delta^\dir$, acting on $L^2(D)$, has a purely discrete spectrum. The eigenvalues $\mu_n^\dir$ can be arranged in a nondecreasing order, counting multiplicities, and they are described by the Courant--Fischer formula~\eqref{eq:cf}, where $H = L^2(\R^d)$. Since $\dom(\form^\dir)$ is contained in $L^2(D)$, we can replace $H = L^2(\R^d)$ by $L^2(D)$:
\formula{
 \mu_n^\dir & = \min \bigl\{ \max \{\form^\dir(u) \, : \, u \in E , \, \lVert u \rVert_2^2 = 1\} \; : \\
 & \qquad\qquad \text{$E$ is an $n$-dimensional subspace of $L^2(D)$} \bigr\} .
}
Let $\ph_n^\dir$ denote the corresponding complete orthonormal system of eigenfunctions of $L^\dir$ on $L^2(D)$. For further discussion, see Section~XIII.15 in~\cite{rs:4}.

If $D$ is a bounded Lipschitz open set, then also $L^\neu = -\Delta^\neu$, restricted to $L^2(D)$, has a purely discrete spectrum. We denote by $\mu_n^\neu$ the eigenvalues of this restriction of $L^\neu$, arranged in a nondecreasing order and counting multiplicities. The Courant--Fischer theorem applies in this case, with $H = L^2(D)$:
\formula{
 \mu_n^\neu & = \min \bigl\{ \max \{\form^\neu(u) \, : \, u \in E , \, \lVert u \rVert_2^2 = 1\} \; : \\
 & \qquad\qquad \text{$E$ is an $n$-dimensional subspace of $L^2(D)$} \bigr\} .
}
We denote by $\ph_n^\neu$ the corresponding complete orthonormal system of eigenfunctions of $L^\neu$ on $L^2(D)$. Again, we refer to Section~XIII.15 in~\cite{rs:4} for additional information.

The subordinate operators $L^{\dir\sub}$ and $L^{\neu\sub}$ have, by definition, the same eigenfunctions as $L^\dir$ and $L^\neu$, with corresponding eigenvalues $\psi(\mu_n^\dir)$ and $\psi(\mu_n^\neu)$, respectively. These eigenvalues are again described by the Courant--Fischer theorem:
\formula*[eq:cf:dirsub]{
 \psi(\mu_n^\dir) & = \min \bigl\{ \max \{\form^{\dir\sub}(u) \, : \, u \in E , \, \lVert u \rVert_2^2 = 1\} \; : \\
 & \qquad\qquad \text{$E$ is an $n$-dimensional subspace of $L^2(D)$} \bigr\} ,
}
and
\formula*[eq:cf:neusub]{
 \psi(\mu_n^\neu) & = \min \bigl\{ \max \{\form^{\neu\sub}(u) \, : \, u \in E , \, \lVert u \rVert_2^2 = 1\} \; : \\
 & \qquad\qquad \text{$E$ is an $n$-dimensional subspace of $L^2(D)$} \bigr\} .
}
Note that if $\psi(\infty)$ is finite, then $\dom(\form^{\dir\sub}) = L^2(\R^d)$ is not contained in $L^2(D)$, but it is easy to see that nevertheless the infimum in~\eqref{eq:cf:dirsub} is attained when $E$ is a subspace of $L^2(D)$.

If $\psi(\infty) = \infty$, then, by a compactness argument fully analogous to the one mentioned above, also the Dirichlet subordinate operator $L^{\sub\dir} = (\psi(-\Delta))^\dir$ has a purely discrete spectrum. By $\lambda_n$ we denote the eigenvalues arranged in a nondecreasing order, and once again the Courant--Fischer theorem says that
\formula*[eq:cf:subdir]{
 \lambda_n & = \min \bigl\{ \max \{\form^{\sub\dir}(u) \, : \, u \in E , \, \lVert u \rVert_2^2 = 1\} \; : \\
 & \qquad\qquad \text{$E$ is an $n$-dimensional subspace of $L^2(D)$} \bigr\} .
}
We denote the complete orthonormal system of the corresponding eigenfunctions by $\ph_n^{\sub\dir}$. When $\psi(\infty)$ is finite, the picture is essentially the same, but the eigenvalues $\lambda_n$ converge to $\psi(\infty)$, which is therefore the only point in the essential spectrum of $L^{\sub\dir}$. A detailed discussion can be found in~\cite{cs:2}.

The proof of the weak inequalities in Theorem~\ref{thm:main} is now straightforward. More work is needed to prove the strict inequalities.

\begin{proof}[Proof of Theorem~\ref{thm:main}]
By Lemmas~\ref{lem:ns:s:ds} and~\ref{lem:s:sd:ds}, we have $\form^{\neu\sub} \le \form^{\sub\dir} \le \form^{\dir\sub}$ on $L^2(D)$. The desired inequality~\eqref{eq:main} thus follows from the Courant--Fischer formulae~\eqref{eq:cf:dirsub}, \eqref{eq:cf:neusub} and~\eqref{eq:cf:subdir}.

Suppose that $\psi$ does not extend to a meromorphic function in $\C \setminus \{0\}$, and that we have equality in the former of the inequalities~\eqref{eq:main}, that is, $\lambda_n = \psi(\mu_n^\neu)$ for some $n$. Then $L^{\sub\dir}$ and $L^{\neu\sub}$ have a common eigenvector $u \in L^2(D)$ corresponding to the eigenvalue $\lambda_n = \psi(\mu_n^\neu)$. In particular, $u \in \dom(\form^{\sub\dir})$ and $\form^{\neu\sub}(u) = \form^{\sub\dir}(u) = \form^\sub(u)$. Recall that $L^\neu \le L$. Thus, by Lemma~\ref{lem:om:strict}, $u$ belongs to a closed subspace $H_0$ of $L^2(\R^d)$, which is invariant for $L$ and $L^\neu$, and on which $\form = \form^\neu$. Furthermore, since $u \in L^2(D)$ and the resolvent of $L^\neu$ preserves $L^2(D)$, formula~\eqref{eq:om:strict:cyclic} implies that $H_0$ is contained in $L^2(D)$. However, the only closed invariant subspace of $L$ contained in $L^2(D)$ is trivial, a contradiction.

Strictness of the other inequality in~\eqref{eq:main} is proved in a completely analogous manner.
\end{proof}

\begin{proof}[Proof of Theorem~\ref{thm:interval}]
Estimate~\eqref{eq:interval} is a direct consequence of Theorem~\ref{thm:main} and the fact that for $D = (a, b)$ we have
\formula{
 \mu_n^\dir & = \biggl( \frac{n \pi}{b - a} \biggr)^2 , &
 \mu_n^\neu & = \biggl( \frac{(n - 1) \pi}{b - a} \biggr)^2 .
}
Strictness of these inequalities when $\psi$ does not extend to a meromorphic function in $\C \setminus \{0\}$ also follows from Theorem~\ref{thm:main}.
\end{proof}

%
%

\section{Eigenvalues in a ball}

Suppose that $d \ge 2$ and $D$ is the unit ball in $\R^d$. Fix a (nonzero) solid harmonic polynomial $P$ of degree $\ell = 0, 1, \ldots\,$. In this section we consider the Hilbert space
\formula{
 L_P^2(\R^d) & = \{ u \in L^2(\R^d) \, : \, u(x) = P(x) \tilde u(\lvert x \rvert) \text{ for some function } \tilde u \} ,
}
which is a closed subspace of $L^2(\R^d)$.

By Bochner's relation (Corollary on p.~72 in~\cite{stein}), the space $L_{\smash{P}}^2(\R^d)$ is invariant under the Fourier transform. As a consequence, $L_{\smash{P}}^2(\R^d)$ is an invariant subspace for the Laplace operator $L = -\Delta$ and the subordinate operator $L^\sub = \psi(-\Delta)$. Since $D$ is invariant under rotations, the same applies to the Dirichlet and Neumann Laplace operators $\Delta^\dir$ and $\Delta^\neu$, as well as their subordinate versions. We refer to~\cite{dkk:1} for a more detailed discussion.

Let $c_P$ denote the integral of $\lvert P(\thet) \rvert^2$ with respect to the surface measure on the unit sphere in $\R^d$. We consider the unitary map
\formula{
 U_P : L_{\smash{P}}^2(\R^d) \ni P(x) \tilde u(\lvert x \rvert) \mapsto \sqrt{c_P} \, r^\ell \tilde u(r) \in L^2((0, \infty), r^{d - 1} dr) ,
}
where $L^2((0, \infty), r^{d - 1} dr)$ is the weighted Lebesgue space. Under $U_P$, the action of $L = -\Delta$ on $L_{\smash{P}}^2(\R^d)$ is unitarily equivalent to the action of the operator
\formula[eq:radial:operators]{
 \tilde L_{d, \ell} & = -\frac{d^2}{d r^2} - \frac{d - 1}{r} \, \frac{d}{d r} + \frac{\ell (d + \ell - 2)}{r^2}
}
on $L^2((0, \infty), r^{d - 1} dr)$; we refer to Section~C.3 in~\cite{fls} for a detailed discussion. By~\eqref{eq:form:balakrishnan}, also the subordinate operators $\tilde L_{\smash{d, \ell}}^\sub = \psi(\tilde L_{d, \ell})$ on $L^2((0, \infty), r^{d - 1} dr)$ and $L^\sub = \psi(L)$ on $L_{\smash{P}}^2(\R^d)$ are unitarily equivalent. Hence, the Dirichlet subordinate operators $\tilde L_{\smash{d, \ell}}^{\sub\dir}$ on $L^2([0, 1), r^{d - 1} dr)$ and $L^{\sub\dir}$ on $L^2(D) \cap L_{\smash{P}}^2(\R^d)$, obtained from $\tilde L_{\smash{d, \ell}}^\sub$ and $L^\sub$ by restricting the corresponding quadratic forms to functions with $\tilde u$ equal to zero almost everywhere in $[1, \infty)$, are unitarily equivalent, too.

Let us denote the eigenvalues corresponding to eigenfunctions in $L^2(D) \cap L_{\smash{P}}^2(\R^d)$ by the same symbol as for $L^2(D)$, with an additional index $\ell$ in the subscript. Thus, $\mu_{\smash{\ell, n}}^\dir$ and $\mu_{\smash{\ell, n}}^\neu$ are the $n$th smallest eigenvalues of $-\Delta^\dir$ and $-\Delta^\neu$ on $L^2(D) \cap L_{\smash{P}}^2(\R^d)$, respectively, and $\lambda_{\ell, n}$ is the $n$th smallest Dirichlet eigenvalue of $L^{\sub\dir}$ on $L^2(D) \cap L_{\smash{P}}^2(\R^d)$. Unitary equivalence discussed in the previous paragraph shows that $\mu_{\smash{\ell, n}}^\dir$, $\mu_{\smash{\ell, n}}^\neu$ and $\lambda_{\ell, n}$ are the sequences of eigenvalues of $\tilde L_{\smash{d, \ell}}^\dir$ (the operator $\tilde L_{d, \ell}$ in $(0, 1)$ with a Dirichlet boundary condition at $r = 1$), $\tilde L_{\smash{d, \ell}}^\neu$ (the same operator with the Neumann boundary condition at $r = 1$) and $\tilde L_{\smash{d, \ell}}^{\sub\dir}$, respectively. In particular, these eigenvalues do not depend on the choice of the solid harmonic polynomial $P$ of a given degree $\ell$, and so our notation is unambiguous.

By definition, for a fixed $\ell$, $\lambda_{\ell, n}$ is a nondecreasing function of $n$. Below we argue that it is also a nondecreasing function of $\ell$. By~\eqref{eq:radial:operators}, if the dimension $d$ is fixed, $\tilde L_{d, \ell}$ form a nondecreasing sequence of operators. Operator monotonicity of $\psi$ implies the same property of $\tilde L_{\smash{d, \ell}}^\sub$, and consequently also $\tilde L_{\smash{d, \ell}}^{\sub\dir}$ form a nondecreasing sequence. Hence, the Courant--Fischer formula~\eqref{eq:cf} implies that the corresponding eigenvalues $\lambda_{\ell, n}$ are indeed nondecreasing functions of $\ell$.

We claim that if $\psi$ is not constant, the inequalities between the quadratic forms discussed above are strict for every nonzero function in the domains of both forms. For $\tilde L_{d, \ell}$, this follows from~\eqref{eq:radial:operators}: the difference between the quadratic forms of $\tilde L_{d, \ell + 1}$ and $\tilde L_{d, \ell}$ is the strictly positive diagonal form
\formula{
 f & \mapsto (d + 2 \ell - 1) \int_0^\infty \frac{\lvert f(r) \rvert^2}{r^2} \, r^{d - 1} dr .
}
Strict inequality for the subordinate operators $\tilde L_{\smash{d, \ell}}^\sub$ is obtained using~\eqref{eq:form:balakrishnan} and a modification of Lemma~\ref{lem:rm} described in Remark~\ref{rem:rm:strict}. Finally, by restricting the domain, we obtain a strict inequality for $\tilde L_{\smash{d, \ell}}^{\sub\dir}$. Combining this with the Courant--Fischer formula~\eqref{eq:cf}, we find that $\lambda_{\ell, n}$ is strictly increasing with respect to $\ell$. This is essentially the same argument as in Lemma~C.5 in~\cite{fls}, but applied to subordinate operators and arbitrary $n$.

Finally, if $\psi$ is not constant, then standard arguments show that the resolvent of $\tilde L_{\smash{d, \ell}}^{\sub\dir}$ is positivity improving in $(0, 1)$, and so the Perron--Frobenius theorem implies that $\lambda_{\ell, 1} < \lambda_{\ell, 2}$; this is analogous to Lemma~C.4 in~\cite{fls}. We can summarise the above considerations as follows: $\lambda_{\ell, n}$ is always nondecreasing with respect to both $\ell$ and $n$, and if $\psi$ is not constant, then in fact
\formula[eq:monotonicity]{
 \begin{matrix}
  \lambda_{0, 1} & < & \lambda_{0, 2} & \le & \lambda_{0, 3} & \le & \cdots \\
  \rotatebox[origin=c]{-90}{$<$} & & \rotatebox[origin=c]{-90}{$<$} & & \rotatebox[origin=c]{-90}{$<$} & & \\
  \lambda_{1, 1} & < & \lambda_{1, 2} & \le & \lambda_{1, 3} & \le & \cdots \\
  \rotatebox[origin=c]{-90}{$<$} & & \rotatebox[origin=c]{-90}{$<$} & & \rotatebox[origin=c]{-90}{$<$} & & \\
  \lambda_{2, 1} & < & \lambda_{2, 2} & \le & \lambda_{2, 3} & \le & \cdots \\
  \rotatebox[origin=c]{-90}{$<$} & & \rotatebox[origin=c]{-90}{$<$} & & \rotatebox[origin=c]{-90}{$<$} & & \\
  \vdots & & \vdots & & \vdots & &
 \end{matrix}
}

Bochner's relation additionally states that the action of the $d$-dimensional Fourier transform on radial profiles $\tilde u$ of functions in $L_{\smash{P}}^2(\R^d)$ only depends on $d + 2 \ell$, save for the normalisation constant. Hence, if we denote by $\mu_{\ell, n}^\dir(d)$ and $\mu_{0, n}^\dir(d + 2 \ell)$ the Dirichlet eigenvalues in the unit balls in $\R^d$ and $\R^{d + 2 \ell}$, respectively, then
\formula{
 \mu_{\ell, n}^\dir(d) & = \mu_{0, n}^\dir(d + 2 \ell) .
}
We refer to~\cite{dkk:2} for further details. Note that a similar relation does not hold for the Neumann eigenvalues, because the Neumann boundary condition in dimension $d$ translates into a Robin boundary condition in dimension $d + 2 \ell$, not the Neumann one.

Among the various relations between $\mu_{\ell, n}^\dir$ and $\mu_{\ell, n}^\neu$, we will need the following one:
\formula[eq:dir:neu]{
 \mu_{0, n + 1}^\neu & = \mu_{1, n}^\dir
}
for $n = 1, 2, \ldots$ (and we only use $n = 1$). This identity follows easily from the observation that the partial derivative with respect to $x_1$ maps non-constant radial (so corresponding to $\ell = 0$ and $P(x) = 1$) Neumann eigenfunctions in the unit ball into antisymmetric (corresponding to $\ell = 1$ and $P(x) = x_1$) Dirichlet eigenfunctions in the unit ball. A more direct proof uses the fact that $(\mu_{\smash{\ell, n}}^\dir)^{1 / 2}$ are the positive zeroes of $t^{1 - d / 2} J_{d / 2 + \ell - 1}(t)$, while $(\mu_{\smash{\ell, n}}^\neu)^{1 / 2}$ are the positive (if $\ell \ge 1$) or nonnegative (if $\ell = 0$) critical points of $t^{1 - d / 2} J_{d / 2 + \ell - 1}(t)$. The property of Bessel functions $J_\nu$ given in \nist{10.6}{6} in~\cite{nist} implies~\eqref{eq:dir:neu}.

\begin{proof}[Proof of Theorem~\ref{thm:ball}]
The argument that we used to prove Theorem~\ref{thm:main} applies verbatim to the restriction of the Laplace operator $\Delta$ to $L_{\smash{P}}^2(\R^d)$. Formula~\eqref{eq:ball} then corresponds to~\eqref{eq:main}.
\end{proof}

\begin{proof}[Proof of Corollary~\ref{cor:shape}]
By monotonicity of $\lambda_{\ell, n}$ with respect to both $\ell$ and $n$, $\lambda_{0, 1}$ is the least eigenvalue, and the second smallest one (counting multiplicities) is necessarily $\lambda_{0, 2}$ or $\lambda_{1, 1}$. It is thus sufficient to prove that $\lambda_{0, 2} \ge \lambda_{1, 1}$. But this is an immediate consequence of Theorem~\ref{thm:ball} and~\eqref{eq:dir:neu}:
\formula{
 \lambda_{0, 2} & \ge \psi(\mu_{0, 2}^{\neu}) = \psi(\mu_{1, 1}^\dir) \ge \lambda_{1, 1} .
}
If $\psi$ does not extend to a meromorphic function in $\C \setminus \{0\}$, the above inequalities are strict: $\lambda_{0, 2} > \lambda_{1, 1}$. In this case $\psi$ is also non-constant, and so, by~\eqref{eq:monotonicity}, $\lambda_{0, 1}$ is the unique smallest eigenvalue, and the eigenspace corresponding to the second smallest eigenvalue $\lambda_{1, 1}$ contains only antisymmetric functions.
\end{proof}

%
%

\section*{}

\subsection*{Acknowledgements}

We thank Kamil Kaleta for an inspiring discussion on Dirichlet--Neumann bracketing.

The large language model \emph{GPT-6 Astra} (OpenAI) was used to assist with reviewing the mathematical arguments, checking references, and editing earlier drafts of this article.

This research was funded in whole or in part by National Science Centre, Poland, grant number 2023/49/B/ST1/04303. For the purpose of Open Access, the authors have applied a CC-BY public copyright licence to any Author Accepted Manuscript (AAM) version arising from this submission.

%
%

%
%


\begin{thebibliography}{00}

\bibitem{atw}
\textsc{Wolfgang Arendt, Antonius Frederik Maria ter Elst, Mahamadi Warma,}
\newblock
\textit{Fractional powers of sectorial operators via the Dirichlet-to-Neumann operator.}
\newblock
Comm.\@ Partial Differ.\@ Equ.\@ 43(1) (2018): 1--24.
\doi{10.1080/03605302.2017.1363229}

\bibitem{ah}
\textsc{Sigurd Assing, John Herman,}
\newblock
\textit{Extension technique for functions of diffusion operators: a stochastic approach.}
\newblock
Electron.\@ J.\@ Probab.\@ 26 (2021), no.\@ 67: 1--32.
\doi{10.1214/21-EJP624}

\bibitem{bk}
\textsc{Rodrigo Bañuelos, Tadeusz Kulczycki,}
\newblock
\textit{The Cauchy process and the Steklov problem.}
\newblock
J.\@ Funct.\@ Anal.\@ 211 (2004): 355--423.
\doi{10.1016/j.jfa.2004.02.005}

\bibitem{bbdg}
\textsc{Jiří Benedikt, Vladimir Bobkov, Raj Narayan Dhara, Petr Girg,}
\newblock
\textit{Nonradiality of second eigenfunctions of the fractional Laplacian in a ball.}
\newblock
Proc.\@ Amer.\@ Math.\@ Soc.\@ 150 (2022): 5335--5348.
\doi{10.1090/proc/16062}

\bibitem{cs:1}
\textsc{Luis Caffarelli, Luis Silvestre,}
\newblock
\textit{An extension problem related to the fractional Laplacian.}
\newblock
Comm.\@ Partial Differ.\@ Equ.\@ 32 (2007): 1245--1260.
\doi{10.1080/03605300600987306}

\bibitem{cs:2}
\textsc{Zhen-Qing Chen, Renming Song,}
\newblock
\textit{Two-sided eigenvalue estimates for subordinate processes in domains.}
\newblock
J.~Funct.\@ Anal.\@ 226(1) (2005): 90--113.
\doi{10.1016/j.jfa.2005.05.004}

\bibitem{ch}
\textsc{Richard Courant, David Hilbert,}
\newblock
\textit{Methods of Mathematical Physics, Vol.\@ I.}
\newblock
Wiley, 1989.
\doi{10.1002/9783527617210}

\bibitem{deblassie}
\textsc{R. Dante DeBlassie,}
\newblock
\textit{Higher order PDEs and symmetric stable processes.}
\newblock
Probab. Theory Relat. Fields 129 (2004): 495--536.
\doi{10.1007/s00440-004-0347-x}

\bibitem{dfw}
\textsc{Sidy Moctar Djitte, Mouhamed Moustapha Fall, Tobias Weth,}
\newblock
\textit{A generalized fractional Pohozaev identity and applications.}
\newblock
Adv.\@ Calc.\@ Var.\@ 17(1) (2024): 237--253.
\doi{10.1515/acv-2022-0003}

\bibitem{dkk:1}
\textsc{Bartłomiej Dyda, Alexey Kuznetsov, Mateusz Kwaśnicki,}
\newblock
\textit{Fractional Laplace operator and Meijer G-function.}
\newblock
Constructive Approx.\@ 45(3) (2017): 427--448.
\doi{10.1007/s00365-016-9336-4}

\bibitem{dkk:2}
\textsc{Bartłomiej Dyda, Alexey Kuznetsov, Mateusz Kwaśnicki,}
\newblock
\textit{Eigenvalues of the fractional Laplace operator in the unit ball.}
\newblock
J.~London Math.\@ Soc.\@ 95 (2017): 500--518.
\doi{10.1112/jlms.12024}

\bibitem{fftw}
\textsc{Mouhamed Moustapha Fall, Pierre Aime Feulefack, Remi Yvant Temgoua, Tobias Weth,}
\newblock
\textit{Morse index versus radial symmetry for fractional Dirichlet problems.}
\newblock
Adv.\@ Math.\@ 384 (2021), no.\@ 107728: 1--22.
\doi{10.1016/j.aim.2021.107728}

\bibitem{fgmp}
\textsc{Mouhamed Moustapha Fall, Marco Ghimenti, Anna Maria Micheletti, Angela Pistoia,}
\newblock
\textit{Generic properties of eigenvalues of the fractional Laplacian.}
\newblock
Calc.\@ Var.\@ Partial Differ.\@ Equ.\@ 62 (2023), no.\@ 233: 1--17.
\doi{10.1007/s00526-023-02574-8}

\bibitem{ferreira}
\textsc{Rui A. C. Ferreira,}
\newblock
\textit{Anti-symmetry of the second eigenfunction of the fractional Laplace operator in a 3-D ball.}
\newblock
Nonlinear Differ.\@ Equ.\@ Appl.\@ 26 (2019), no.\@ 6: 1--8.
\doi{10.1007/s00030-019-0554-x}

\bibitem{fls}
\textsc{Rupert L. Frank, Enno Lenzmann, Luis Silvestre,}
\newblock
\textit{Uniqueness of radial solutions for the fractional Laplacian.}
\newblock
Comm.\@ Pure Appl.\@ Math.\@ 69(9) (2016): 1671--1726.
\doi{10.1002/cpa.21591}

\bibitem{fot}
\textsc{Masatoshi Fukushima, Yoichi Oshima, Masayoshi Takeda,}
\newblock
\textit{Dirichlet Forms and Symmetric Markov Processes.}
\newblock
Second revised and extended edition. De Gruyter, Berlin-New York, 2011.
\doi{10.1515/9783110218091}

\bibitem{gms}
\textsc{José E.\@ Galé, Pedro J.\@ Miana, Pablo Raúl Stinga,}
\newblock
\textit{Extension problem and fractional operators: semigroups and wave equations.}
\newblock
J.\@ Evol.\@ Equ.\@ 13(2) (2013): 343--368.
\doi{10.1007/s00028-013-0182-6}

\bibitem{hssw}
\textsc{Seppo Hassi, Adrian Sandovici, Hendrik S.\,V.\@ de Snoo, Henrik Winkler,}
\newblock
\textit{Form sums of nonnegative selfadjoint operators.}
\newblock
Acta Math.\@ Hung.\@ 111(1--2) (2006): 81--105.
\doi{10.1007/s10474-006-0036-6}

\bibitem{kkms}
\textsc{Tadeusz Kulczycki, Mateusz Kwaśnicki, Jacek Małecki, Andrzej Stós,}
\newblock
\textit{Spectral properties of the Cauchy process on half-line and interval.}
\newblock
Proc.\@ London Math.\@ Soc.\@ 101(2) (2010): 589--622.
\doi{10.1112/plms/pdq010}

\bibitem{kwasnicki:1}
\textsc{Mateusz Kwaśnicki,}
\newblock
\textit{Eigenvalues of the Cauchy process on an interval have at most double multiplicity.}
\newblock
Semigroup Forum 79(1) (2009): 183--192.
\doi{10.1007/s00233-009-9166-9}

\bibitem{kwasnicki:2}
\textsc{Mateusz Kwaśnicki,}
\newblock
\textit{Eigenvalues of the fractional Laplace operator in the interval.}
\newblock
J.\@ Funct.\@ Anal.\@ 262(5) (2012): 2379--2402.
\doi{10.1016/j.jfa.2011.12.004}

\bibitem{kwasnicki}
\textsc{Mateusz Kwaśnicki,}
\newblock
\textit{Simplicity of eigenvalues of the fractional Laplace operator in an interval.}
\newblock
Unpublished, 2023.
\arxiv{2311.00713}

\bibitem{km}
\textsc{Mateusz Kwaśnicki, Jacek Mucha,}
\newblock
\textit{Extension technique for complete Bernstein functions of the Laplace operator.}
\newblock
J.~Evol.\@ Equ.\@ 18(3) (2018): 1341--1379.
\doi{10.1007/s00028-018-0444-4}

\bibitem{lee}
\textsc{David Lee,}
\newblock
\textit{An extension problem for higher order operators and operators of logarithmic type via renormalization.}
\newblock
J.\@ Math.\@ Anal.\@ Appl.\@ 559(2) (2026), no.\@ 130510: 1--27.
\doi{10.1016/j.jmaa.2026.130510}

\bibitem{ms}
\textsc{Celso Martínez Carracedo, Miguel Sanz Alix,}
\newblock
\textit{The Theory of Fractional Powers of Operators.}
\newblock
North-Holland Math.\@ Studies 187, Elsevier, Amsterdam, 2001.
\isbn{978-0-444-88797-9}

\bibitem{mn:1}
\textsc{Roberta Musina, Alexander I.\@ Nazarov,}
\newblock
\textit{On Fractional Laplacians.}
\newblock
Comm.\@ Partial Differ.\@ Equ.\@ 39(9) (2014): 1780--1790.
\doi{10.1080/03605302.2013.864304}

\bibitem{mn:2}
\textsc{Roberta Musina, Alexander I.\@ Nazarov,}
\newblock
\textit{On Fractional Laplacians -- 2.}
\newblock
Ann.\@ Inst.\@ Henri Poincaré C Anal.\@ Non Linéaire 33(6) (2016): 1667--1673.
\doi{10.1016/j.anihpc.2015.08.001}

\bibitem{nazarov}
\textsc{Alexander I.\@ Nazarov,}
\newblock
\textit{On comparison of fractional Laplacians.}
\newblock
Nonlinear Anal.\@ 218 (2022), no.\@ 112790: 1--7.
\doi{10.1016/j.na.2022.112790}

\bibitem{ouhabaz}
\textsc{El Maati Ouhabaz,}
\newblock
\textit{Analysis of Heat Equations on Domains.}
\newblock
Princeton University Press, 2005.
\isbn{9780691120164}

\bibitem{rs:1}
\textsc{Michael Reed, Barry Simon,}
\newblock
\textit{Methods of Modern Mathematical Physics, Vol.\@ 1: Functional Analysis.}
\newblock
Academic Press, 1980.
\isbn{9780125850506}

\bibitem{rs:4}
\textsc{Michael Reed, Barry Simon,}
\newblock
\textit{Methods of Modern Mathematical Physics, Vol.\@ 4: Analysis of Operators.}
\newblock
Academic Press, 1978.
\isbn{9780080570457}

\bibitem{ssv}
\textsc{René L.\@ Schilling, Renming Song, Zoran Vondraček,}
\newblock
\textit{Bernstein Functions: Theory and Applications.}
\newblock
Third revised and extended edition. De Gruyter, 2026.
\doi{10.1515/9783111295121}

\bibitem{stein}
\textsc{Elias M.\@ Stein,}
\newblock
\textit{Singular integrals and differentiability properties of functions.}
\newblock
Princeton University Press, Princeton, 1970.
\doi{10.1515/9781400883882}

\bibitem{st}
\textsc{Pablo Raúl Stinga, José Luis Torrea,}
\newblock
\textit{Extension problem and Harnack's inequality for some fractional operators.}
\newblock
Comm.\@ Partial Differ.\@ Equ.\@ 35 (2010): 2092--2122. 
\doi{10.1080/03605301003735680}

\bibitem{tan}
\textsc{Jinggang Tan,}
\newblock
\textit{The Brezis--Nirenberg type problem involving the square root of the Laplacian.}
\newblock
Calc.\@ Var.\@ Partial Differ.\@ Equ.\@ 42(1--2) (2011): 21--41.
\doi{10.1007/s00526-010-0378-3}

\bibitem{vv}
\textsc{Hendrik Vogt, Jürgen Voigt,}
\newblock
\textit{Increasing sequences of sectorial forms.}
\newblock
Czech.\@ Math.~J.\@ 70(4) (2020): 1033--1046.
\doi{10.21136/CMJ.2020.0101-19}

\bibitem{nist}
\textit{NIST Digital Library of Mathematical Functions,}
\href{https://dlmf.nist.gov/}{\textsf{dlmf.nist.gov}}\textsf{,}
Release 1.2.4 of 2025-03-15.
\textsc{F.\,W.\,J.\@ Olver, A.\,B.\@ Olde Daalhuis, D.\,W.\@ Lozier, B.\,I.\@ Schneider, R.\,F.\@ Boisvert, C.\,W.\@ Clark, B.\,R.\@ Miller, B.\,V.\@ Saunders, H.\,S.\@ Cohl,} and \textsc{M.\,A.\@ McClain,} eds.

\end{thebibliography}
\end{document}